\documentclass[11pt]{article}

\usepackage{amsmath,amssymb,amsthm}
\usepackage[a4paper,textwidth=15cm,top=3cm,bottom=3cm]{geometry}
\usepackage{url}
\usepackage[hidelinks]{hyperref}
\hypersetup{pdftitle={Arbitrary Stage Spectra for Raikov Remainders Along Quotient Towers}, pdfauthor={Xing-Yu Hu}, pdfsubject={General topology and topological groups}, pdfkeywords={topological group, Raikov completion, Raikov remainder, quotient homomorphism, quotient tower, pseudocompactness, Boolean group, stage spectrum}}

\newcommand{\doi}[1]{doi: \url{#1}}

\newtheorem{theorem}{Theorem}[section]
\newtheorem{proposition}[theorem]{Proposition}
\newtheorem{corollary}[theorem]{Corollary}
\theoremstyle{definition}
\newtheorem{example}[theorem]{Example}

\title{Arbitrary Stage Spectra for Raikov Remainders Along Quotient Towers}
\newcommand{\shorttitle}{Stage Spectra of Raikov Remainders}

\author{Xing-Yu Hu\thanks{Corresponding author.}\\
{\small School of Mathematics and Statistics, Hanjiang Normal University}\\
{\small Shiyan, Hubei 442000, China}\\
{\small \texttt{huxingyu@hjnu.edu.cn}}}
\date{}

\begin{document}
\pagestyle{myheadings}
\markboth{X.-Y. Hu}{\shorttitle}
\maketitle

\begin{abstract}
Let $\theta$ be an ordinal and let $(N_\alpha)_{\alpha\leq\theta}$ be a continuous increasing tower of closed normal subgroups of a Hausdorff topological group $G$, with $N_0=\{e\}$. Here continuity means that each limit term is the closure in $G$ of the union of all preceding terms. Write $\rho G$ for the Raikov completion of $G$ and put $H_\alpha=G\overline{N_\alpha}^{\rho G}$. The stage spectrum consists of the ordinals $\gamma\in(0,\theta]$ at which some point of $\rho G\setminus G$ first enters $H_\gamma$. For every nonzero ordinal $\theta$ and every $S\subseteq(0,\theta]$, there is a Hausdorff pseudocompact Boolean group $G$ with a continuous tower of closed subgroups, each pseudocompact, whose stage spectrum is exactly $S$. The construction can be arranged so that $G/N_\theta\cong\mathbb Z/2\mathbb Z$ and, for every $\delta<\theta$, the factor $N_{\delta+1}/N_\delta$ is compact exactly when $\delta+1\notin S$. In particular, the stage spectrum can consist of a single limit ordinal even when every successive factor is compact.

For any Hausdorff topological group $G$ and closed normal subgroup $N$, let $\widehat q:\rho G\to\rho(G/N)$ extend the quotient homomorphism and put $K=\overline N^{\rho G}$. The part of the fiber over $gN\in G/N$ that lies in $\rho G\setminus G$ is $g(K\setminus N)$, homeomorphic to $\rho N\setminus N$. For $y\in\widehat q(\rho G)\setminus(G/N)$, the entire fiber lies in $\rho G\setminus G$ and is homeomorphic to $\rho N$. When $N$ is precompact, these fibers give a decomposition of $\rho G\setminus G$. For nested quotients, the preimage formula gives a two-stage identity. If $N$ is compact, the remainders $\rho G\setminus G$ and $\rho(G/N)\setminus(G/N)$ are pseudocompact simultaneously. In two examples with countable precompact Abelian groups $G$ and precompact but noncompact kernels $N$, the remainder $\rho G\setminus G$ is pseudocompact in both cases, whereas $\rho(G/N)\setminus(G/N)$ is pseudocompact in only one.
\end{abstract}



\section{Introduction}

All topological groups are Hausdorff. The empty space is regarded as pseudocompact. Each group is viewed as a dense subgroup of its Raikov completion via the canonical embedding.

For a Hausdorff topological group $G$, let $\rho G$ denote its Raikov completion and put
\[
 r_\rho(G)=\rho G\setminus G.
\]
This completion remainder is called the Raikov remainder of $G$. Earlier work on remainders of topological groups includes Arhangel'skii's pseudocompact--Lindel\"of dichotomy for arbitrary compactification remainders and subsequent results on remainder duality and completeness-type properties \cite{Arhangelskii2008Remainders,Arhangelskii2009Remainders,ArhangelskiiChoban2019}. Arhangel'skii and Choban studied Raikov remainders in the broader setting of group extensions. For a dense subgroup $G$ of a topological group $H$, the set $H\setminus G$ is called a group-remainder of $G$. In particular, they showed that every nonempty group-remainder is dense in its ambient group \cite[Proposition~1.2]{ArhangelskiiChoban2018}.

Let $N$ be a closed normal subgroup of $G$, let $q:G\to G/N$ be the quotient homomorphism, and let $\widehat q:\rho G\to\rho(G/N)$ be its continuous homomorphic extension \cite[Corollary~3.6.17]{ArhangelskiiTkachenko2008}. The map $\widehat q$ is open onto its image, and its kernel is $\overline N^{\rho G}$ \cite[Theorem~3.6.19]{ArhangelskiiTkachenko2008}. This kernel identity determines which points of $\rho G\setminus G$ map into the original quotient group $G/N$. If $N$ is precompact, then $\widehat q$ is onto \cite[Proposition~3.7.19]{ArhangelskiiTkachenko2008}. Surjectivity need not hold for arbitrary closed normal subgroups. Leischner gave an example of an incomplete quotient of a complete topological group and later studied quotients of Raikov-complete groups \cite{Leischner1991,Leischner1993}.

Put $K=\overline N^{\rho G}$. Since $K$ is a Raikov completion of $N$, the fiber of $\widehat q$ over $gN\in G/N$ meets $r_\rho(G)$ in $g(K\setminus N)$, a translated copy of $r_\rho(N)$. If $y\in\widehat q(\rho G)\setminus(G/N)$, the entire fiber over $y$ lies in $r_\rho(G)$ and is homeomorphic to $\rho N$. Thus the part of the completion outside $G$ that maps into $G/N$ depends on whether $N$ is Raikov complete. In the complete case these fibers are empty. Otherwise they are nonempty over every point of $G/N$.

For a precompact kernel, the completion extension is automatically onto. The fiber dichotomy then gives a more detailed decomposition. The set $r_\rho(G)$ splits into points mapped by $\widehat q$ into $G/N$ and points mapped into $r_\rho(G/N)$. The first part has fibers homeomorphic to $r_\rho(N)$, while the second part has fibers homeomorphic to $\rho N$. When neither $N$ nor $G/N$ is Raikov complete, both parts are dense in $\rho G$ by the classical density theorem for nonempty group-remainders \cite[Proposition~1.2]{ArhangelskiiChoban2018}.

The decomposition extends to nested quotients. For a closed normal subgroup $N$, set $D_G(N)=\widehat q_N^{-1}(G/N)\setminus G$, where $\widehat q_N$ is the completion extension for the quotient by $N$. Thus $D_G(N)$ consists of points outside $G$ that map into $G/N$. For nested closed normal subgroups $N\subseteq M$, functoriality of the completion extensions and the preimage formula in Theorem~\ref{thm:image-dichotomy} give a two-stage identity. The set $D_G(M)$ splits into points already belonging to $D_G(N)$ and points that enter after the further quotient by $M/N$. Iterating this identity along a normal quotient tower starting from the trivial subgroup gives a transfinite first-entry stratification. Successor strata come from the successive quotient maps. At a limit stage, new points may enter upon taking the closure in $\rho G$ of the union of the earlier subgroup closures. The identity concerns remainder fibers, not exactness of the Raikov-completion functor.

Every subset of the nonzero stages occurs as the stage spectrum of a continuous tower in a pseudocompact Boolean group, with every member of the tower pseudocompact. For these towers, a successive factor is compact precisely when its stage is absent from the prescribed spectrum. Peng studied a different, cardinal-valued quotient-weight spectrum for precompact Abelian groups and showed that such spectra can be arbitrary subsets of a cardinal interval, with pseudocompact realizations in part of that range \cite{Peng2023}. Peng and Zhou study density spectra associated with dense and closed subgroups of topological groups \cite{Peng2026}.

Dense subgroups of compact Boolean products are standard ingredients in constructions of pseudocompact groups. A closely related classical construction is due to Comfort and Soundararajan. For a prime $p$ and an uncountable cardinal $\alpha$, they obtain a dense pseudocompact subgroup of $(\mathbb Z/p\mathbb Z)^\alpha$ of index $p$ by starting from the countable-support subgroup $\Sigma$ \cite[Lemma~4.2]{ComfortSoundararajan1982}. Galindo, Tkachenko, Bruguera, and Hern\'andez use a $\Sigma$-product to exhibit a pseudocompact group with a closed $\omega$-bounded subgroup and a compact metrizable quotient \cite[Corollary~4.7]{GalindoTkachenkoBrugueraHernandez2014}. Under $2^{\omega_1}=\mathfrak c$, Leiderman and Tkachenko construct pseudocompact Abelian groups whose closed subgroups are separable \cite[Theorem~3.4]{LeidermanTkachenko2018}. Gabriyelyan, van Mill, and Reznichenko likewise use $\Sigma$-product methods in pseudocompact-group constructions \cite[Theorem~5.1]{GabriyelyanVanMillReznichenko2026}. They also state the standard Boolean-group fact that maps from independent subsets to $\mathbb Z/2\mathbb Z$ extend to homomorphisms \cite[Fact~4.2]{GabriyelyanVanMillReznichenko2026}. The proof of Theorem~\ref{thm:arbitrary-spectrum} uses an algebraic homomorphism from a compact Boolean product to encode the prescribed first-entry stages. The kernel contains the countable-support subgroup, and its intersections with coordinate subproducts are used to construct the pseudocompact tower subgroups.

Related work concerns the behavior of completeness and completion under quotients. Morales and Tkachenko study strong realcompactness and strong Dieudonn\'e completeness, defined by closed embeddings into products of separable metrizable groups and of metrizable groups, respectively. Their quotient results require a \v{C}ech-complete normal subgroup, together with $\omega$-narrowness of the ambient group in the first case or $\omega$-balancedness in the second \cite[Theorem~4.5]{MoralesTkachenko2016}. For an Abelian group $G$ and a closed subgroup $N$ whose Raikov completion is \v{C}ech-complete, Bello, Chasco, Dom\'inguez, and Tkachenko prove that $\rho(G/N)\cong\rho G/\overline N^{\rho G}$ as topological groups \cite[Proposition~3.9]{BelloChascoDominguezTkachenko2016}. Under the same hypothesis, they also prove that $1\to\rho N\to\rho G\to\rho(G/N)\to1$ is an extension of topological Abelian groups \cite[Proposition~3.10]{BelloChascoDominguezTkachenko2016}. Arnautov and Ermakova study completeness of quotients by discrete, compact, and locally compact subgroups \cite{ArnautovErmakova2024,ArnautovErmakova2025}. Braunling and Ren prove that completion preserves exact sequences in the category of locally precompact Hausdorff Abelian groups \cite[Proposition~6.5]{BraunlingRen2025}. The present paper studies the fibers of $\widehat q$ on $r_\rho(G)$ and first-entry stages along quotient towers.

For compact $N$, the restriction $\widehat q|_{r_\rho(G)}:r_\rho(G)\to r_\rho(G/N)$ is an open perfect surjection, so these two remainders are pseudocompact simultaneously. The perfectness of quotient maps by compact subgroups is classical in the theory of coset spaces \cite{Arhangelskii2017}. Arhangel'skii and Choban also use open perfect restrictions arising from compact-subgroup quotients in their study of group-remainders \cite{ArhangelskiiChoban2018}. Their later work treats properties of group-remainders and their perfect images in a broader mapping setting \cite{ArhangelskiiChoban2021}.

If $r_\rho(G)$ is nonempty and pseudocompact, its density in $\rho G$ forces $\rho G$ to be pseudocompact. The compactness theorem for Raikov completions of pseudocompact groups then implies that $\rho G$ is compact \cite[Corollary~3.7.18]{ArhangelskiiTkachenko2008}. Hence $G$ is precompact, and the precompact-kernel dichotomy applies to every closed normal subgroup of $G$.

Examples~\ref{ex:preserve} and~\ref{ex:destroy} use the Bohr topology on $\mathbb Z$ and a result of Arhangel'skii and Bella stated in the abstract of their paper \cite{ArhangelskiiBella2001}. In that abstract, a group is called pseudocompact at infinity when every compactification remainder is pseudocompact. In both examples, $\widehat q(r_\rho(G))=\rho(G/N)$, whereas $r_\rho(G/N)$ is pseudocompact in the first example and not in the second.

\section{Quotient Fibers and Canonical Decomposition}

\subsection{Fiber structure and the quotient image dichotomy}

For the quotient homomorphism $q:G\to G/N$, let
\[
 \widehat q:\rho G\longrightarrow\rho(G/N)
\]
be its extension, and put
\[
 K=\overline N^{\rho G}
 \qquad\text{and}\qquad
 I=\widehat q(\rho G).
\]
By the classical completion theorem, $\widehat q$ is open onto $I$ and $\ker\widehat q=K$ \cite[Theorem~3.6.19]{ArhangelskiiTkachenko2008}. Since $K$ is a closed subgroup of the Raikov-complete group $\rho G$, it is Raikov complete \cite[Exercise~3.6.m]{ArhangelskiiTkachenko2008}. The subgroup $N$ is dense in $K$, and the subgroup-uniformity and uniqueness results show that this inclusion realizes $K$ as a Raikov completion of $N$ \cite[Proposition~1.8.4 and Theorem~3.6.14]{ArhangelskiiTkachenko2008}.

\begin{theorem}\label{thm:image-dichotomy}
Let $N$ be a closed normal subgroup of a topological group $G$. With $q$, $K$, $I$, and $\widehat q$ as above,
\[
 \widehat q^{-1}(G/N)=GK.
\]
The restriction of $\widehat q$ to $r_\rho(G)$ has the following fiber structure.
\begin{enumerate}
\item If $y=q(g)\in G/N$, then
\[
 \bigl(\widehat q|_{r_\rho(G)}\bigr)^{-1}(y)
 =g(K\setminus N),
\]
which is homeomorphic to $r_\rho(N)$.
\item If $y\in I\setminus(G/N)$ and $\widehat q(x)=y$, then
\[
 \bigl(\widehat q|_{r_\rho(G)}\bigr)^{-1}(y)=xK,
\]
which is homeomorphic to $\rho N$.
\end{enumerate}
Consequently, the image is given by the following two cases.
\begin{enumerate}
\item If $N$ is Raikov complete, then
\[
 \widehat q\bigl(r_\rho(G)\bigr)=I\setminus(G/N).
\]
\item If $N$ is not Raikov complete, then
\[
 \widehat q\bigl(r_\rho(G)\bigr)=I.
\]
\end{enumerate}
In particular,
\[
 r_\rho(G/N)=\widehat q\bigl(r_\rho(G)\bigr)
\]
holds if and only if $N$ is Raikov complete and $\widehat q$ is onto.
\end{theorem}

\begin{proof}
By the classical facts above, $\ker\widehat q=K$ and $K$ is a Raikov completion of $N$. If $x\in\widehat q^{-1}(G/N)$, choose $g\in G$ such that $\widehat q(x)=q(g)$. Then $g^{-1}x\in K$, so $x\in GK$. The reverse inclusion follows from $\widehat q(K)=\{e\}$. Hence
\[
 \widehat q^{-1}(G/N)=GK.
\]
Because $N$ is closed in $G$,
\[
 K\cap G=N.
\]

Fix $y=q(g)\in G/N$. The full $\widehat q$-fiber over $y$ is $gK$. Its intersection with $G$ is $gN$. Therefore
\[
 \bigl(\widehat q|_{r_\rho(G)}\bigr)^{-1}(y)
 =gK\setminus gN
 =g(K\setminus N).
\]
Left translation by $g$ is a homeomorphism, and under the completion isomorphism $K\cong\rho N$ fixing $N$, the set $K\setminus N$ corresponds to $r_\rho(N)$.

Now let $y\in I\setminus(G/N)$ and choose $x\in\rho G$ with $\widehat q(x)=y$. The full fiber is $xK$. If $xk\in G$ for some $k\in K$, then
\[
 y=\widehat q(xk)\in G/N,
\]
a contradiction. Hence $xK\subseteq r_\rho(G)$, and the restricted fiber equals $xK$, which is homeomorphic to $K$ and therefore to $\rho N$.

If $N$ is Raikov complete, then $K=N$ and the fibers over points of $G/N$ are empty. Hence $\widehat q(r_\rho(G))=I\setminus(G/N)$. If $N$ is not Raikov complete, then $K\setminus N$ is nonempty. Every point of $G/N$ is then in $\widehat q(r_\rho(G))$, and every point of $I\setminus(G/N)$ is covered by the second fiber formula. Hence $\widehat q(r_\rho(G))=I$.

The final equivalence follows from the two cases. When $N$ is Raikov complete, $\widehat q(r_\rho(G))=r_\rho(G/N)$ exactly when $I=\rho(G/N)$. When $N$ is not Raikov complete, the image contains $G/N$, so it cannot equal $r_\rho(G/N)$.
\end{proof}

If $N$ is precompact, then $\widehat q$ is onto \cite[Proposition~3.7.19]{ArhangelskiiTkachenko2008}, and a precompact group is Raikov complete exactly when it is compact \cite[Theorem~3.7.15]{ArhangelskiiTkachenko2008}.

\begin{corollary}\label{cor:precompact-kernel}
Let $N$ be a closed normal precompact subgroup of a topological group $G$, and let $\widehat q:\rho G\to\rho(G/N)$ be the completion extension of the quotient homomorphism. Then
\[
 \widehat q\bigl(r_\rho(G)\bigr)=
 \begin{cases}
 r_\rho(G/N),&\text{if $N$ is compact},\\[4pt]
 \rho(G/N),&\text{if $N$ is not compact}.
 \end{cases}
\]
In particular,
\[
 r_\rho(G/N)=\widehat q\bigl(r_\rho(G)\bigr)
\]
if and only if $N$ is compact.
\end{corollary}

\begin{proof}
Apply Theorem~\ref{thm:image-dichotomy} using the two classical facts above.
\end{proof}

\subsection{Canonical decomposition for precompact kernels}

For a precompact kernel, the fiber theorem determines the image of $r_\rho(G)$ under the extended quotient map and a canonical intermediate group between $G$ and its Raikov completion.

Let $N$ be a closed normal precompact subgroup of $G$, put
\[
 Q=G/N,
 \qquad
 K=\overline N^{\rho G},
\]
and let $q:G\to Q$ be the quotient homomorphism and $\widehat q:\rho G\to\rho Q$ its completion extension. The map $\widehat q$ is onto \cite[Proposition~3.7.19]{ArhangelskiiTkachenko2008}. Since $K$ is a Raikov completion of the precompact group $N$, it is compact \cite[Theorem~3.7.15]{ArhangelskiiTkachenko2008}. Define
\[
 H_N=\widehat q^{-1}(Q)=GK.
\]
Set
\[
 L_N=H_N\setminus G
 \qquad\text{and}\qquad
 T_N=\rho G\setminus H_N.
\]
Thus $L_N$ consists of points of $r_\rho(G)$ mapped into $Q$ by $\widehat q$, whereas $T_N$ is the full inverse image of $r_\rho(Q)$.

\begin{theorem}\label{thm:canonical-decomposition}
Let $N$ be a closed normal precompact subgroup of a topological group $G$. With the notation above, the following assertions hold.
\begin{enumerate}
\item The group $H_N$ is a dense intermediate subgroup
\[
 G\subseteq H_N\subseteq\rho G,
\]
and $\rho G$ is a Raikov completion of $H_N$. With this choice of completion, $\rho H_N=\rho G$. The compact group $K$ is normal in $H_N$, and
\[
 H_N/K\cong Q.
\]
\item The Raikov remainder admits the canonical decomposition
\[
 r_\rho(G)=L_N\mathbin{\dot\cup}T_N,
\]
where
\[
 L_N=H_N\setminus G=GK\setminus G
\]
and
\[
 T_N=\rho G\setminus H_N=r_\rho(H_N).
\]
\item One has $H_N=G$ if and only if $N$ is compact. One has $H_N=\rho G$ if and only if $Q$ is Raikov complete.
\item If $N$ is not compact, then
\[
 \widehat q|_{L_N}:L_N\longrightarrow Q
\]
is a continuous open surjection. Every fiber is homeomorphic to $r_\rho(N)$.
\item If $Q$ is not Raikov complete, then
\[
 \widehat q|_{T_N}:T_N\longrightarrow r_\rho(Q)
\]
is an open perfect surjection. Every fiber is homeomorphic to $\rho N$.
\end{enumerate}
\end{theorem}

\begin{proof}
The equality $H_N=GK$ follows from Theorem~\ref{thm:image-dichotomy}. Since $G\subseteq H_N\subseteq\rho G$ and $G$ is dense in $\rho G$, the group $H_N$ is dense in $\rho G$. The two-sided uniformity on a subgroup is the restriction of the ambient two-sided uniformity \cite[Proposition~1.8.4]{ArhangelskiiTkachenko2008}. Hence $\rho G$ is a Raikov completion of $H_N$ \cite[Theorem~3.6.14]{ArhangelskiiTkachenko2008}. With this choice, $\rho H_N=\rho G$.

The kernel of $\widehat q|_{H_N}$ is $K$, and the restriction is an open surjection from $H_N$ onto $Q$. Thus $H_N/K\cong Q$. The definitions now give
\[
 L_N=H_N\setminus G
\]
and
\[
 T_N=\rho G\setminus H_N=r_\rho(H_N),
\]
which proves the decomposition of $r_\rho(G)$.

By Theorem~\ref{thm:image-dichotomy}, the set $L_N$ is empty exactly when the precompact group $N$ is Raikov complete. This is equivalent to compactness of $N$. The set $T_N$ is empty exactly when $r_\rho(Q)$ is empty, which is equivalent to Raikov completeness of $Q$. This proves the two endpoint characterizations.

Suppose that $N$ is not compact. The fiber formula over $Q$ in Theorem~\ref{thm:image-dichotomy} gives
\[
 \bigl(\widehat q|_{L_N}\bigr)^{-1}(q(g))
 =g(K\setminus N)
 \cong r_\rho(N).
\]
It remains to check openness. Let $U=O\cap L_N$, where $O$ is open in $\rho G$. If $y\in\widehat q(O)\cap Q$, then the full fiber over $y$ is a coset $gK$ meeting $O$. The group $N$ is a proper dense subgroup of $K$. A proper dense subgroup has empty interior, so $K\setminus N$ is dense in $K$. Hence the nonempty open set $O\cap gK$ meets $g(K\setminus N)$. Therefore
\[
 \widehat q(U)=\widehat q(O)\cap Q,
\]
which is open in $Q$.

For assertion (5), assume that $Q$ is not Raikov complete. The open surjection $\widehat q$ has compact kernel $K$, so the induced topological isomorphism $(\rho G)/K\cong\rho Q$ and the compact-subgroup quotient theorem show that $\widehat q$ is perfect \cite[Theorem~1.5.7]{ArhangelskiiTkachenko2008}. The set $T_N$ is the full inverse image of $r_\rho(Q)$. The restriction of an open perfect map to the full inverse image of any subspace is again open and perfect. Its fibers are cosets of $K\cong\rho N$.
\end{proof}

If $N$ is not compact, then $L_N=H_N\setminus G$ is a nonempty group-remainder of $G$ in $H_N$ and is therefore dense in $H_N$ \cite[Proposition~1.2]{ArhangelskiiChoban2018}. Since $H_N$ is dense in $\rho G$, the set $L_N$ is dense in $\rho G$. If $Q$ is not Raikov complete, then $T_N=r_\rho(H_N)$ is a nonempty group-remainder and is dense in $\rho H_N=\rho G$ by the same theorem.

For precompact $N$, these conclusions give four cases. If $N$ is compact and $Q$ is Raikov complete, then $r_\rho(G)$ is empty. If exactly one of these two conditions fails, then $r_\rho(G)$ equals precisely one of $L_N$ and $T_N$. If $N$ is not compact and $Q$ is not Raikov complete, then
\[
 r_\rho(G)=L_N\mathbin{\dot\cup}T_N
\]
with both pieces dense in $\rho G$ and hence in $r_\rho(G)$. In the last case, $r_\rho(G)$ is therefore resolvable. The quotient map canonically selects the two dense pieces, and its restrictions to them have different fiber types.

\section{Quotient Towers and Stage Spectra}

\subsection{A two-stage identity for nested quotients}

For a closed normal subgroup $N$ of $G$, let $q_N:G\to G/N$ be the quotient homomorphism and let
\[
 \widehat q_N:\rho G\longrightarrow\rho(G/N)
\]
be its extension. Write
\[
 K_N=\overline N^{\rho G}
\]
and define
\[
 D_G(N)=\widehat q_N^{-1}(G/N)\setminus G=GK_N\setminus G.
\]
Thus $D_G(N)$ consists of the points of $r_\rho(G)$ whose images under $\widehat q_N$ belong to $G/N$. By Theorem~\ref{thm:image-dichotomy}, the set $D_G(N)$ is empty if and only if $N$ is Raikov complete.

\begin{theorem}\label{thm:tower-law}
Let $N\subseteq M$ be closed normal subgroups of a topological group $G$. Use the canonical topological isomorphism
\[
 (G/N)/(M/N)\cong G/M
\]
and regard the second quotient as $G/M$. Let
\[
 q_{M/N}:G/N\longrightarrow G/M
\]
be the induced quotient homomorphism and put
\[
 K_{M/N}=\overline{M/N}^{\,\rho(G/N)}.
\]
Let
\[
 a=\widehat q_N:\rho G\to\rho(G/N)
\]
and let
\[
 b=\widehat q_{M/N}:\rho(G/N)\to\rho(G/M)
\]
be the completion extension of $q_{M/N}$. Then the extension of the direct quotient by $M$ is $b\circ a$. Also,
\[
 a^{-1}\bigl(K_{M/N}\bigr)=K_M
\]
and
\[
 a^{-1}\bigl(D_{G/N}(M/N)\bigr)
 =GK_M\setminus GK_N.
\]
Consequently,
\[
 D_G(M)
 =D_G(N)\mathbin{\dot\cup}
 a^{-1}\bigl(D_{G/N}(M/N)\bigr).
\]
\end{theorem}

\begin{proof}
The quotient homomorphisms satisfy
\[
 q_M=q_{M/N}\circ q_N.
\]
Both $\widehat q_M$ and $b\circ a$ are continuous homomorphisms from $\rho G$ to the Hausdorff group $\rho(G/M)$, and each restricts to $q_M$ on the dense subgroup $G$. Hence they coincide, so
\[
 \widehat q_M=b\circ a.
\]
The kernel-closure theorem gives
\[
 a^{-1}(K_{M/N})
 =a^{-1}(\ker b)
 =\ker(b\circ a)
 =\ker\widehat q_M
 =K_M.
\]

Applying Theorem~\ref{thm:image-dichotomy} to the second quotient gives
\[
 b^{-1}(G/M)=(G/N)K_{M/N}.
\]
Taking the inverse image under $a$ and using $b\circ a=\widehat q_M$ gives
\[
 a^{-1}\bigl((G/N)K_{M/N}\bigr)
 =\widehat q_M^{-1}(G/M)
 =GK_M.
\]
For the first quotient, Theorem~\ref{thm:image-dichotomy} gives
\[
 a^{-1}(G/N)=GK_N.
\]
Since inverse images commute with set difference,
\[
\begin{aligned}
 a^{-1}\bigl(D_{G/N}(M/N)\bigr)
 &=a^{-1}\bigl((G/N)K_{M/N}\setminus(G/N)\bigr)\\
 &=GK_M\setminus GK_N.
\end{aligned}
\]
The remaining set difference decomposes as
\[
 GK_M\setminus G
 =\bigl(GK_N\setminus G\bigr)
 \mathbin{\dot\cup}
 \bigl(GK_M\setminus GK_N\bigr),
\]
which is the displayed formula.
\end{proof}

The identity distinguishes points of $D_G(M)$ according to whether their images under $\widehat q_N$ belong to $G/N$. For the second term, these images lie in $r_\rho(G/N)$ but map into $G/M$ under the further quotient by $M/N$.

\subsection{Transfinite stratification}

\begin{theorem}\label{thm:transfinite-tower}
Let $\theta$ be an ordinal, let $G$ be a topological group, and let $(N_\alpha)_{\alpha\leq\theta}$ be an increasing chain of closed normal subgroups of $G$, where $N_0=\{e\}$. Put
\[
 K_\alpha=\overline{N_\alpha}^{\rho G},
 \qquad
 H_\alpha=GK_\alpha.
\]
For $0<\gamma\leq\theta$, define the $\gamma$-th first-entry stratum by
\[
 S_\gamma=
 \begin{cases}
 H_{\delta+1}\setminus H_\delta,&\text{if $\gamma=\delta+1$},\\[3pt]
 H_\gamma\setminus\displaystyle\bigcup_{\alpha<\gamma}H_\alpha,
     &\text{if $\gamma$ is a limit ordinal}.
 \end{cases}
\]
Then, for every $\beta\leq\theta$,
\[
 D_G(N_\beta)=\mathbin{\dot\bigcup}_{0<\gamma\leq\beta}S_\gamma.
\]
At a successor stage $\gamma=\delta+1$, if
\[
 a_\delta:\rho G\longrightarrow\rho(G/N_\delta)
\]
denotes the completion extension, then
\[
 S_{\delta+1}
 =a_\delta^{-1}\bigl(D_{G/N_\delta}(N_{\delta+1}/N_\delta)\bigr).
\]
If $\lambda\leq\theta$ is a limit ordinal and the chain is continuous there in the sense that
\[
 N_\lambda=\overline{\bigcup_{\alpha<\lambda}N_\alpha}^{\,G},
\]
then
\[
 K_\lambda
 =\overline{\bigcup_{\alpha<\lambda}K_\alpha}^{\,\rho G}.
\]

If, in addition, $N_\theta$ is precompact, then for every $\delta<\theta$,
\[
 S_{\delta+1}=\varnothing
 \quad\Longleftrightarrow\quad
 N_{\delta+1}/N_\delta\text{ is compact}.
\]
\end{theorem}

\begin{proof}
The groups $H_\alpha$ form an increasing chain and $H_0=G$. Fix $x\in D_G(N_\beta)=H_\beta\setminus G$. The set of ordinals $\gamma\leq\beta$ for which $x\in H_\gamma$ has a least member, say $\gamma_x>0$. If $\gamma_x=\delta+1$, then $x\in H_{\delta+1}\setminus H_\delta=S_{\gamma_x}$. If $\gamma_x$ is a limit ordinal, then $x$ belongs to $H_{\gamma_x}$ and to none of the preceding groups, so $x\in S_{\gamma_x}$. This gives the disjoint stratification.

The successor formula is Theorem~\ref{thm:tower-law} applied to $N_\delta\subseteq N_{\delta+1}$. At a continuous limit stage,
\[
\begin{aligned}
 K_\lambda
 &=\overline{N_\lambda}^{\rho G}
 =\overline{\bigcup_{\alpha<\lambda}N_\alpha}^{\rho G}
 =\overline{\bigcup_{\alpha<\lambda}K_\alpha}^{\rho G},
\end{aligned}
\]
which proves the limit-stage assertion.

Suppose in addition that $N_\theta$ is precompact, and fix $\delta<\theta$. Then $N_\delta$ and $N_{\delta+1}/N_\delta$ are precompact. Hence $a_\delta$ is onto \cite[Proposition~3.7.19]{ArhangelskiiTkachenko2008}. By the successor formula,
\[
 S_{\delta+1}=\varnothing
 \quad\Longleftrightarrow\quad
 D_{G/N_\delta}(N_{\delta+1}/N_\delta)=\varnothing.
\]
By Theorem~\ref{thm:image-dichotomy}, the latter condition is equivalent to Raikov completeness of $N_{\delta+1}/N_\delta$, and precompactness makes this equivalent to compactness \cite[Theorem~3.7.15]{ArhangelskiiTkachenko2008}. This proves the final assertion.
\end{proof}

For $0<\gamma\leq\theta$, put $P_\gamma=\bigcup_{\alpha<\gamma}H_\alpha$. Because $G\subseteq P_\gamma\subseteq H_\gamma\subseteq\rho G$ and $G$ is dense in $\rho G$, the group $P_\gamma$ is dense in $H_\gamma$ and $H_\gamma$ is dense in $\rho G$. By definition,
\[
 S_\gamma=H_\gamma\setminus P_\gamma.
\]
Hence each nonempty first-entry stratum $S_\gamma$ is a group-remainder of $P_\gamma$ in $H_\gamma$. By the general density theorem for nonempty group-remainders \cite[Proposition~1.2]{ArhangelskiiChoban2018}, it is dense in $H_\gamma$, and therefore in $\rho G$.

Assume again that $N_\theta$ is precompact. For every $\delta<\theta$, the open surjection $a_\delta$ has compact kernel $K_\delta$. The induced topological isomorphism $(\rho G)/K_\delta\cong\rho(G/N_\delta)$ and the compact-subgroup quotient theorem show that $a_\delta$ is perfect \cite[Theorem~1.5.7]{ArhangelskiiTkachenko2008}. By the successor formula, $S_{\delta+1}$ is the full inverse image of $D_{G/N_\delta}(N_{\delta+1}/N_\delta)$ under $a_\delta$. Hence the restriction
\[
 a_\delta|_{S_{\delta+1}}:S_{\delta+1}\longrightarrow
 D_{G/N_\delta}(N_{\delta+1}/N_\delta)
\]
is an open perfect surjection whose fibers are homeomorphic to the compact group $K_\delta\cong\rho N_\delta$. By contrast, limit strata need not be determined by the compactness of the successive factors.

\subsection{Arbitrary stage spectra}

For a tower as in Theorem~\ref{thm:transfinite-tower}, define its stage spectrum by
\[
 \Sigma(N_\bullet)=\{\gamma\in(0,\theta]:S_\gamma\ne\varnothing\}.
\]
Under the precompactness hypothesis in Theorem~\ref{thm:transfinite-tower}, compactness determines whether a successor stratum is empty. At a limit stage, a stratum may be nonempty even when every successive factor is compact.

In this subsection, Boolean groups are written additively.

\begin{theorem}\label{thm:arbitrary-spectrum}
Let $\theta$ be a nonzero ordinal and let $S\subseteq(0,\theta]$. There exist a Hausdorff pseudocompact Boolean group $G$ and a continuous increasing chain $(N_\alpha)_{\alpha\leq\theta}$ of closed normal subgroups of $G$, with $N_0=\{0\}$, such that every $N_\alpha$ is pseudocompact and
\[
 \Sigma(N_\bullet)=S.
\]
For this tower,
\[
 r_\rho(G)=\mathbin{\dot\bigcup}_{\gamma\in S}S_\gamma.
\]
The construction can be arranged so that
\[
 G/N_\theta\cong\mathbb Z/2\mathbb Z.
\]
For every $\delta<\theta$,
\[
 N_{\delta+1}/N_\delta\text{ is compact}
 \quad\Longleftrightarrow\quad
 \delta+1\notin S.
\]
If $S\ne\varnothing$, then $G$ and $N_\theta$ are not Raikov complete.
\end{theorem}

\begin{proof}
If $S=\varnothing$, take $G=\mathbb Z/2\mathbb Z$ and $N_\alpha=\{0\}$ for all $\alpha\leq\theta$. Then $r_\rho(G)=\varnothing$, and all assertions follow. Hence assume $S\ne\varnothing$.

The construction uses a compact Boolean product $K$ and its countable-support subgroup $D$. Since $D$ is $G_\delta$-dense in $K$ and meets each coordinate subproduct in a $G_\delta$-dense subgroup, the groups constructed below are pseudocompact. A linear map on $K/D$ determines which first-entry stages occur.

Let
\[
 I=\theta\times S\times\omega_1,
 \qquad
 K=2^I,
\]
where $2=\mathbb Z/2\mathbb Z$, and equip $K$ with the compact product topology. For $\alpha\leq\theta$, put
\[
 C_\alpha=\{x\in K:x(\xi,s,\eta)=0\text{ whenever }\xi\geq\alpha\}.
\]
Then $C_0=\{0\}$, $C_\theta=K$, and for every nonzero limit $\lambda\leq\theta$,
\[
 C_\lambda=
 \overline{\bigcup_{\alpha<\lambda}C_\alpha}^{\,K}.
\]

Let
\[
 D=\{x\in K:|\operatorname{supp}(x)|\leq\omega\}.
\]
Thus $D$ is the countable-support subgroup of $K$. It is $G_\delta$-dense in $K$. Let $Z\subseteq K$ be a nonempty $G_\delta$ set and write $Z=\bigcap_{n<\omega}O_n$, where each $O_n$ is open in $K$. Choose $z\in Z$. For every $n$, choose a finite set $J_n\subseteq I$ such that the cylinder of all points agreeing with $z$ on $J_n$ is contained in $O_n$. If $J=\bigcup_{n<\omega}J_n$, the point that agrees with $z$ on $J$ and is zero off $J$ belongs to $D\cap Z$. By the Comfort--Ross characterization of pseudocompact totally bounded groups, $D$ is pseudocompact \cite[Theorem~1.2]{ComfortRoss1966}.

Regard $K/D$ as a vector space over $\mathbb F_2$, write $q:K\to K/D$ for the algebraic quotient map, and put $U_\alpha=q(C_\alpha)$. For each $\gamma\in S$, choose $x_\gamma\in C_\gamma$ as follows. If $\gamma=\delta+1$, let
\[
 \operatorname{supp}(x_\gamma)
 =\{(\delta,\gamma,\eta):\eta<\omega_1\}.
\]
If $\gamma$ is a nonzero limit ordinal, let
\[
 \operatorname{supp}(x_\gamma)
 =\{(\xi,\gamma,\eta):\xi<\gamma,\ \eta<\omega_1\}.
\]
Write $\bar x_\gamma=q(x_\gamma)$. For every $\alpha<\gamma$ and every $y\in C_\alpha$, the difference $x_\gamma-y$ has uncountably many nonzero coordinates whose first coordinate is at least $\alpha$. Thus $x_\gamma-y\notin D$, so
\[
 \bar x_\gamma\in U_\gamma
 \setminus\bigcup_{\alpha<\gamma}U_\alpha.                 \tag{3.1}
\]
The supports of the $x_\gamma$ are pairwise disjoint in their second coordinate. Thus every nonempty finite sum of them has uncountable support. Consequently
\[
 \{\bar x_\gamma:\gamma\in S\}
\]
is linearly independent in $K/D$.

By transfinite recursion, choose bases $\mathcal B_\alpha$ of $U_\alpha$ such that
\[
 \mathcal B_\alpha\subseteq\mathcal B_\beta
 \quad(\alpha\leq\beta\leq\theta)
\]
and
\[
 \bar x_\gamma\in\mathcal B_\alpha
 \quad\text{whenever }\gamma\in S\text{ and }\gamma\leq\alpha.       \tag{3.2}
\]
At a successor stage, this follows from (3.1) and basis extension. At a limit stage, the union of the earlier bases spans $\bigcup_{\beta<\alpha}U_\beta$ and remains independent. If $\alpha\in S$, equation (3.1) allows $\bar x_\alpha$ to be adjoined, after which the resulting independent set can be extended to a basis of $U_\alpha$.

Let
\[
 V=\bigoplus_{\gamma\in S}\mathbb F_2 e_\gamma.
\]
Define a linear map $\pi:K/D\to V$ on the basis $\mathcal B_\theta$ by
\[
 \pi(\bar x_\gamma)=e_\gamma\quad(\gamma\in S)
\]
and by $\pi(b)=0$ for every other $b\in\mathcal B_\theta$. Put
\[
 \phi=\pi\circ q:K\longrightarrow V,
 \qquad
 G_0=\ker\phi.
\]
The argument uses $\phi$ only as an algebraic homomorphism. Since $D\subseteq G_0$ and $D$ is $G_\delta$-dense in $K$, the group $G_0$ is $G_\delta$-dense in $K$ and hence pseudocompact by the same Comfort--Ross theorem. In particular, $G_0$ is dense in $K$. The subgroup uniformity is the restriction of the ambient two-sided uniformity, so $K$ is a Raikov completion of $G_0$ \cite[Proposition~1.8.4 and Theorem~3.6.14]{ArhangelskiiTkachenko2008}. Under the completion isomorphism extending the inclusion $G_0\subseteq K$, write $\rho G_0=K$.

For $\alpha\leq\theta$, set
\[
 M_\alpha=G_0\cap C_\alpha.
\]
The subgroup $M_\alpha$ is closed in $G_0$. The subgroup $D\cap C_\alpha$ is $G_\delta$-dense in $C_\alpha$ by the same cylinder argument used for $D$. Since it is contained in $M_\alpha$, the subgroup $M_\alpha$ is $G_\delta$-dense in $C_\alpha$ and is pseudocompact by \cite[Theorem~1.2]{ComfortRoss1966}. In particular,
\[
 \overline{M_\alpha}^{\,K}=C_\alpha.                         \tag{3.3}
\]
If $\lambda\leq\theta$ is a nonzero limit ordinal, every finite-support point of $C_\lambda$ belongs to some $C_\alpha$ with $\alpha<\lambda$. Such points lie in $\bigcup_{\alpha<\lambda}M_\alpha$ and are dense in $C_\lambda$. Since $M_\lambda=G_0\cap C_\lambda$ is closed in $G_0$,
\[
 M_\lambda=
 \overline{\bigcup_{\alpha<\lambda}M_\alpha}^{\,G_0}.       \tag{3.4}
\]
Hence $(M_\alpha)_{\alpha\leq\theta}$ is continuous.

By (3.1)--(3.2) and the definition of the linear map above,
\[
 \phi(C_\alpha)=W_\alpha
 :=\operatorname{span}\{e_\gamma:\gamma\in S,\ \gamma\leq\alpha\}.
                                                                    \tag{3.5}
\]
Equation~(3.5) and $G_0=\ker\phi$ give
\[
 G_0+C_\alpha=\phi^{-1}(W_\alpha).                         \tag{3.6}
\]
Write $H_\alpha^0=G_0+C_\alpha$. If $\gamma=\delta+1$, then (3.5)--(3.6) give
\[
 H_\gamma^0\setminus H_\delta^0\ne\varnothing
 \quad\Longleftrightarrow\quad
 W_\gamma\ne W_\delta
 \quad\Longleftrightarrow\quad
 \gamma\in S.                                               \tag{3.7}
\]
If $\gamma$ is a nonzero limit ordinal, every vector in $V$ has finite support, so
\[
 \bigcup_{\alpha<\gamma}W_\alpha
 =\operatorname{span}\{e_s:s\in S,\ s<\gamma\}.
\]
It follows that
\[
 H_\gamma^0\setminus\bigcup_{\alpha<\gamma}H_\alpha^0\ne\varnothing
 \quad\Longleftrightarrow\quad
 \gamma\in S.                                               \tag{3.8}
\]

To obtain the stated terminal quotient, let $L=\mathbb Z/2\mathbb Z$ with its discrete topology and put
\[
 G=G_0\times L,
 \qquad
 N_\alpha=M_\alpha\times\{0\}.
\]
Since $G_0$ is a subgroup of the Boolean group $K$ and $L$ is Boolean, the group $G$ is Boolean. It is Hausdorff as a subgroup of the compact Hausdorff group $K\times L$, and it is pseudocompact because it is a finite union of copies of the pseudocompact group $G_0$. Each $N_\alpha$ is closed in $G$ because $M_\alpha$ is closed in $G_0$, normal because $G$ is Abelian, and pseudocompact because it is homeomorphic to $M_\alpha$. Since $C_0=\{0\}$, one has $N_0=\{0\}$. Moreover, $G$ is dense in $K\times L$, so the same subgroup-uniformity and completion-uniqueness facts show that $K\times L$ is a Raikov completion of $G$. Thus $\rho G$ may be taken to be $K\times L$, and (3.3) gives
\[
 \overline{N_\alpha}^{\,\rho G}
 =\overline{M_\alpha}^{\,K}\times\{0\}
 =C_\alpha\times\{0\}.
\]
The families $(C_\alpha)_{\alpha\leq\theta}$, $(M_\alpha)_{\alpha\leq\theta}$, and $(N_\alpha)_{\alpha\leq\theta}$ are increasing, and (3.4) shows that the tower $(N_\alpha)_{\alpha\leq\theta}$ is continuous. For this tower, the intermediate groups from Theorem~\ref{thm:transfinite-tower} are
\[
 H_\alpha
 =G+(C_\alpha\times\{0\})
 =(G_0+C_\alpha)\times L
 =H_\alpha^0\times L.
\]
Hence (3.7)--(3.8) show that $\Sigma(N_\bullet)=S$. Since $M_\theta=G_0$,
\[
 G/N_\theta\cong L\cong\mathbb Z/2\mathbb Z.
\]
Since $C_\theta=K$, the displayed formula gives $H_\theta=\rho G$. Hence $D_G(N_\theta)=r_\rho(G)$, and Theorem~\ref{thm:transfinite-tower} together with $\Sigma(N_\bullet)=S$ gives
\[
 r_\rho(G)=\mathbin{\dot\bigcup}_{\gamma\in S}S_\gamma.
\]
In particular, $N_\theta=M_\theta\times\{0\}=G_0\times\{0\}$ is a subgroup of the compact group $K\times L$, so it is precompact. The compactness criterion for successor strata in Theorem~\ref{thm:transfinite-tower} now gives
\[
 N_{\delta+1}/N_\delta\text{ is compact}
 \quad\Longleftrightarrow\quad
 \delta+1\notin S
\]
for every $\delta<\theta$. When $S\ne\varnothing$, the map $\phi$ is nonzero, so $G_0$ is a proper dense subgroup of $K$. Hence $G_0$, $N_\theta\cong G_0$, and $G=G_0\times L$ are not Raikov complete.
\end{proof}

For a nonzero limit ordinal $\lambda\leq\theta$, take $S=\{\lambda\}$ in Theorem~\ref{thm:arbitrary-spectrum}. In the resulting continuous tower, the Boolean group $G$ and every $N_\alpha$ are pseudocompact. Its only nonempty stratum satisfies
\[
 S_\lambda=r_\rho(G)\ne\varnothing
\]
and is dense in $\rho G$. Every successive factor is compact. Thus compactness of all successive factors does not preclude a nonempty limit stratum.

\section{Pseudocompactness and Examples}

\subsection{Pseudocompactness consequences}

For a closed normal subgroup $N$ of $G$, write $\widehat q:\rho G\to\rho(G/N)$ for the completion extension of the quotient homomorphism.

By the locally finite characterization of pseudocompactness \cite[Theorem~3.10.22]{Engelking1989}, an open perfect surjection $f:X\to Y$ between Tychonoff spaces preserves pseudocompactness in both directions. The forward implication follows from preservation under continuous images.

For the reverse implication, assume that $Y$ is pseudocompact and let $\mathcal U$ be a locally finite family of nonempty open subsets of $X$. Since $f$ is open, each $f(U)$ is open. Fix $y\in Y$. Compactness of $f^{-1}(y)$ gives an open neighborhood $W$ of the fiber that meets only finitely many members of $\mathcal U$. Since $f$ is closed, there is an open neighborhood $V$ of $y$ with $f^{-1}(V)\subseteq W$. If $V\cap f(U)\ne\varnothing$, then $f^{-1}(V)\cap U\ne\varnothing$, so $U$ meets $W$. Thus there are only finitely many $U\in\mathcal U$ for which $V\cap f(U)\ne\varnothing$.

Moreover, only finitely many members of $\mathcal U$ can have the same image under $f$. Otherwise, a point in that common image would have a compact fiber meeting infinitely many members of the locally finite family $\mathcal U$. Therefore an infinite $\mathcal U$ would produce an infinite locally finite family of nonempty open subsets of $Y$, contradicting pseudocompactness. Hence $\mathcal U$ is finite, and $X$ is pseudocompact.

The perfectness of quotient maps by compact subgroups is classical \cite[Theorem~1.5.7]{ArhangelskiiTkachenko2008}. Combining this fact with Corollary~\ref{cor:precompact-kernel} gives the following corollary.

\begin{corollary}\label{cor:compact-kernel-perfect}
Let $N$ be a compact normal subgroup of a topological group $G$, and let $\widehat q:\rho G\to\rho(G/N)$ be the completion extension of the quotient homomorphism. Then
\[
 \widehat q|_{r_\rho(G)}:r_\rho(G)\longrightarrow r_\rho(G/N)
\]
is an open perfect surjection. Every fiber is homeomorphic to $N$. Consequently,
\[
 r_\rho(G)\text{ is pseudocompact}
 \quad\Longleftrightarrow\quad
 r_\rho(G/N)\text{ is pseudocompact}.
\]
\end{corollary}

\begin{proof}
Since $N$ is compact, $\widehat q$ is onto by \cite[Proposition~3.7.19]{ArhangelskiiTkachenko2008}. As $\rho G$ is Hausdorff, $N$ is also closed in $\rho G$, so $\overline N^{\rho G}=N$. Theorem~\ref{thm:image-dichotomy} now gives
\[
 \widehat q^{-1}(G/N)=G.
\]
Therefore
\[
 r_\rho(G)=\widehat q^{-1}\bigl(r_\rho(G/N)\bigr).
\]
The map $\widehat q$ is open by the classical completion theorem \cite[Theorem~3.6.19]{ArhangelskiiTkachenko2008}. Since it is onto with compact kernel $N$, the induced topological isomorphism $(\rho G)/N\cong\rho(G/N)$ and the compact-subgroup quotient theorem show that $\widehat q$ is perfect \cite[Theorem~1.5.7]{ArhangelskiiTkachenko2008}. Its restriction to this full inverse image remains open and perfect. The fiber formula in Theorem~\ref{thm:image-dichotomy} shows that every fiber is a coset of $N$. The final equivalence follows because open perfect maps preserve pseudocompactness in both directions.
\end{proof}

The classical density theorem for nonempty group-remainders \cite[Proposition~1.2]{ArhangelskiiChoban2018}, together with the compactness theorem for Raikov completions of pseudocompact groups \cite[Corollary~3.7.18]{ArhangelskiiTkachenko2008}, gives the following consequence.

\begin{corollary}\label{cor:pseudocompact-source}
Let $G$ be a topological group. If $r_\rho(G)$ is nonempty and pseudocompact, then $\rho G$ is compact and $G$ is precompact. Consequently, if $N$ is any closed normal subgroup of $G$ and $\widehat q:\rho G\to\rho(G/N)$ is the completion extension of the quotient homomorphism, then
\[
 \widehat q\bigl(r_\rho(G)\bigr)=
 \begin{cases}
 r_\rho(G/N),&\text{if $N$ is compact},\\[4pt]
 \rho(G/N),&\text{if $N$ is not compact}.
 \end{cases}
\]
If $N$ is compact, then Corollary~\ref{cor:compact-kernel-perfect} gives the stronger equivalence between pseudocompactness of $r_\rho(G)$ and $r_\rho(G/N)$.
\end{corollary}

\begin{proof}
The space $r_\rho(G)$ is dense in $\rho G$ by \cite[Proposition~1.2]{ArhangelskiiChoban2018}. Let $f:\rho G\to\mathbb R$ be continuous. Its restriction to $r_\rho(G)$ is bounded. Density gives
\[
 f(\rho G)\subseteq\overline{f(r_\rho(G))},
\]
so $\rho G$ is pseudocompact. The Raikov completion of every pseudocompact topological group is compact \cite[Corollary~3.7.18]{ArhangelskiiTkachenko2008}. Since $\rho G$ is already Raikov complete, it is compact. Therefore $G$ is precompact. Every subgroup of a precompact group is precompact, so Corollary~\ref{cor:precompact-kernel} gives the displayed dichotomy. The compact-kernel assertion follows from Corollary~\ref{cor:compact-kernel-perfect}.
\end{proof}

For small precompact groups, the image dichotomy does not determine whether $r_\rho(G/N)$ is pseudocompact when the kernel is not Raikov complete. Arhangel'skii and Bella state in their abstract that every nonmetrizable topological group of cardinality at most $\omega_1$ is pseudocompact at infinity \cite{ArhangelskiiBella2001}. Combining this statement with the elementary metrizable case gives the following criterion for small precompact groups.

\begin{proposition}\label{prop:small-classification}
Let $H$ be a precompact topological group with $|H|\leq\omega_1$. Then $r_\rho(H)$ is pseudocompact if and only if $H$ is nonmetrizable or Raikov complete.
\end{proposition}

\begin{proof}
If $H$ is nonmetrizable, then the cited result of Arhangel'skii and Bella shows that $H$ is pseudocompact at infinity. Since $\rho H$ is a compactification of the precompact group $H$, the remainder $r_\rho(H)$ is pseudocompact.

If $H$ is Raikov complete, then precompactness makes $H$ compact, so $r_\rho(H)=\varnothing$.

Suppose that $H$ is metrizable and not Raikov complete. Since $H$ is also precompact, its Raikov completion $\rho H$ is compact and metrizable. The group $H$ is a proper dense subgroup of $\rho H$. The remainder $\rho H\setminus H$ is metrizable and noncompact. If it were compact, it would be closed in $\rho H$, forcing $H$ to be open and hence closed, contrary to proper density. A pseudocompact metrizable space is compact. Hence $r_\rho(H)$ is not pseudocompact.
\end{proof}

\subsection{Examples with countable precompact Abelian groups}\label{sec:examples}

For a discrete Abelian group $D$, let $D^\#$ denote $D$ equipped with its Bohr topology. Then $D^\#$ is Hausdorff and precompact \cite[Corollary~9.9.6]{ArhangelskiiTkachenko2008}. Put
\[
 A=\mathbb Z^\#.
\]
Every compact subset of $A$ is finite \cite[Theorem~9.9.30]{ArhangelskiiTkachenko2008}. Since $A$ is infinite and precompact, it is not discrete. If $A$ were metrizable, non-discreteness and first countability at the identity would give a sequence of distinct points converging to the identity. That sequence together with its limit would be an infinite compact subset of $A$, a contradiction. Hence $A$ is nonmetrizable. If $A$ were Raikov complete, precompactness would make it compact \cite[Theorem~3.7.15]{ArhangelskiiTkachenko2008}, again contradicting the finiteness of compact subsets. Thus $A$ is not Raikov complete.

Choose an irrational number $t$ and let
\[
 B=\{e^{2\pi i nt}:n\in\mathbb Z\}
\]
be equipped with the subspace topology inherited from the circle group $\mathbb T$. Since $t$ is irrational, $B$ is dense in $\mathbb T$. Thus $B$ is countable, metrizable, precompact, and not Raikov complete, and $\mathbb T$ is a Raikov completion of $B$. Write $\rho A$ for the Raikov completion of $A$, and use additive notation for $A$ and multiplicative notation for $B$.

\begin{example}\label{ex:preserve}
Let
\[
 G=A\times A,
 \qquad
 N=A\times\{0\}.
\]
Then $G$ is countable, precompact, Abelian, nonmetrizable, and not Raikov complete. Hence $r_\rho(G)\ne\varnothing$, and Proposition~\ref{prop:small-classification} shows that it is pseudocompact. The subgroup $N$ is closed and normal, precompact, and noncompact. Corollary~\ref{cor:precompact-kernel} gives
\[
 \widehat q\bigl(r_\rho(G)\bigr)=\rho(G/N).
\]
On the other hand,
\[
 G/N\cong A,
\]
so $r_\rho(G/N)$ is pseudocompact by Proposition~\ref{prop:small-classification}. The decomposition in Theorem~\ref{thm:canonical-decomposition} becomes
\[
 L_N=(\rho A\setminus A)\times A
 \qquad\text{and}\qquad
 T_N=\rho A\times(\rho A\setminus A).
\]
Both sets are dense in $\rho A\times\rho A$ \cite[Proposition~1.2]{ArhangelskiiChoban2018}, and Theorem~\ref{thm:canonical-decomposition}(5) shows that $T_N$ maps onto $r_\rho(A)$ by an open perfect map.
\end{example}

\begin{example}\label{ex:destroy}
Let
\[
 G=A\times B,
 \qquad
 N=A\times\{1\}.
\]
Again $G$ is countable, precompact, Abelian, nonmetrizable, and not Raikov complete. Hence $r_\rho(G)\ne\varnothing$, and Proposition~\ref{prop:small-classification} shows that it is pseudocompact. The subgroup $N$ is closed and normal, precompact, and noncompact, so
\[
 \widehat q\bigl(r_\rho(G)\bigr)=\rho(G/N)
\]
by Corollary~\ref{cor:precompact-kernel}. Since
\[
 G/N\cong B,
\]
the quotient is metrizable and not Raikov complete. Proposition~\ref{prop:small-classification}, applied to $G/N$, shows that $r_\rho(G/N)$ is not pseudocompact.

Equivalently,
\[
 r_\rho(G/N)\cong\mathbb T\setminus B.
\]
This space is metrizable and noncompact. If it were compact, it would be closed in $\mathbb T$, so the proper dense subgroup $B$ would be open and therefore closed, a contradiction. Here the canonical decomposition is
\[
 L_N=(\rho A\setminus A)\times B
 \qquad\text{and}\qquad
 T_N=\rho A\times(\mathbb T\setminus B).
\]
Again both parts are dense in $\rho G$ \cite[Proposition~1.2]{ArhangelskiiChoban2018}. The difference from Example~\ref{ex:preserve} is reflected in $r_\rho(G/N)$. By Theorem~\ref{thm:canonical-decomposition}(5), the restriction $\widehat q|_{T_N}:T_N\to\mathbb T\setminus B$ is an open perfect surjection. Hence $T_N$ is not pseudocompact.
\end{example}

In both examples the kernel is precompact and not Raikov complete, so $\widehat q(r_\rho(G))=\rho(G/N)$. What differs is the quotient group. In Example~\ref{ex:preserve} it is nonmetrizable, while in Example~\ref{ex:destroy} it is metrizable and not Raikov complete.

\begin{example}
Let $A=\mathbb Z^\#$ and let $B$ be the dense cyclic subgroup of $\mathbb T$ used above. Put
\[
 G=A\times A\times B,
\]
and consider the closed normal chain
\[
 N_0=\{0\}\times\{0\}\times\{1\}
 \subseteq
 N_1=A\times\{0\}\times\{1\}
 \subseteq
 N_2=A\times A\times\{1\}.
\]
The terminal subgroup $N_2\cong A\times A$ is precompact, and both successive factors $N_1/N_0\cong A$ and $N_2/N_1\cong A$ are noncompact. The compactness criterion in Theorem~\ref{thm:transfinite-tower} therefore gives two nonempty disjoint first-entry strata
\[
 S_1=GK_1\setminus G,
 \qquad
 S_2=GK_2\setminus GK_1,
\]
and
\[
 D_G(N_2)=S_1\mathbin{\dot\cup}S_2.
\]
Each $S_i$ is dense in $\rho G$ \cite[Proposition~1.2]{ArhangelskiiChoban2018}. The final quotient satisfies
\[
 G/N_2\cong B,
\]
so its Raikov remainder is not pseudocompact. Meanwhile, $G$ is countable, precompact, nonmetrizable, and not Raikov complete. Hence $r_\rho(G)\ne\varnothing$, and Proposition~\ref{prop:small-classification} shows that it is pseudocompact. Thus two nonempty strata occur in $D_G(N_2)$ even though $r_\rho(G/N_2)$ is not pseudocompact.
\end{example}

Even when an Abelian group is Raikov complete, one of its quotients may have a non-pseudocompact Raikov remainder. In Problem~7.9.B, Arhangel'skii and Tkachenko ask the reader to show that every Abelian topological group is a quotient of a Raikov-complete Abelian topological group \cite[Problem~7.9.B]{ArhangelskiiTkachenko2008}. Applying this statement to $\mathbb Q$ with its usual topology gives a Raikov-complete Abelian group $H$ and a closed normal subgroup $N$ such that $H/N\cong\mathbb Q$. Since $\mathbb R$ is a Raikov completion of $\mathbb Q$, one has $r_\rho(H)=\varnothing$ while $r_\rho(H/N)$ is homeomorphic to $\mathbb R\setminus\mathbb Q$, which is not pseudocompact. In contrast, Example~\ref{ex:destroy} has $r_\rho(G)$ nonempty and pseudocompact while $r_\rho(G/N)$ is not pseudocompact.

The closure of a normal subgroup in $\rho G$ determines the fibers of the quotient extension that meet $r_\rho(G)$. For precompact kernels, these fibers give the canonical decomposition $r_\rho(G)=L_N\mathbin{\dot\cup}T_N$. For $N\subseteq M$, the two-stage identity decomposes $D_G(M)$ into $D_G(N)$ and $\widehat q_N^{-1}(D_{G/N}(M/N))$. Its transfinite iteration describes the first-entry strata along a normal tower.

When the terminal subgroup is precompact, a successor stratum is empty exactly when the corresponding factor is compact. Limit strata behave differently. Every subset of the nonzero stages is realized as the stage spectrum of a continuous tower $(N_\alpha)$ in a pseudocompact Boolean group $G$, with every $N_\alpha$ pseudocompact and $G/N_\theta\cong\mathbb Z/2\mathbb Z$. For compact $N$, pseudocompactness of $r_\rho(G)$ and $r_\rho(G/N)$ is equivalent. In the examples with countable precompact Abelian groups and noncompact kernels, $r_\rho(G)$ is pseudocompact and $\widehat q(r_\rho(G))=\rho(G/N)$. The remainder $r_\rho(G/N)$ is pseudocompact in one example and not in the other.

\section{Further Questions}

The results above lead to three questions.

\begin{enumerate}
\item Let $N$ be a closed normal precompact noncompact subgroup of $G$. Find useful conditions under which the set
\[
 L_N=\widehat q^{-1}(G/N)\setminus G
\]
is pseudocompact or countably compact. In particular, determine to what extent these properties can be inferred from $G/N$ and $r_\rho(N)$. The restriction $\widehat q|_{L_N}:L_N\to G/N$ is a continuous open surjection, and every fiber is homeomorphic to $r_\rho(N)$. Since $N$ is precompact and noncompact, $\rho N$ is compact and $N$ is a proper dense subgroup. Thus $r_\rho(N)$ is noncompact. If it were compact, it would be closed in $\rho N$, making $N$ an open subgroup and hence closed, contrary to proper density. Therefore the map is not perfect.

\item Characterize continuous chains of closed normal subgroups for which every limit stratum in Theorem~\ref{thm:transfinite-tower} is empty. Theorem~\ref{thm:arbitrary-spectrum} shows that compactness of all successive factors is not sufficient. Such a criterion should account for the closure of $\bigcup_{\alpha<\lambda}K_\alpha$ in $\rho G$ at each limit stage $\lambda$, in addition to continuity of the subgroup chain inside $G$.

\item Does there exist a topological group $G$ and a closed normal subgroup $N$ such that $r_\rho(G)$ is nonempty and countably compact while $r_\rho(G/N)$ is not countably compact? Any such example must have $N$ noncompact. Countable compactness of $r_\rho(G)$ implies pseudocompactness, so Corollary~\ref{cor:pseudocompact-source} makes $G$ precompact. If $N$ were compact, Corollary~\ref{cor:compact-kernel-perfect} would make $r_\rho(G/N)$ a continuous image of $r_\rho(G)$, and hence countably compact. Consequently, $G/N$ is precompact and $\rho(G/N)$ is compact in any such example. Arhangel'skii and Bella state in their abstract that pseudocompactness at infinity need not imply countable compactness at infinity even for countable groups \cite{ArhangelskiiBella2001}.
\end{enumerate}

\end{document}